\documentclass[letterpaper, 10 pt, conference]{ieeeconf}                                                         
\IEEEoverridecommandlockouts                          
\usepackage[noadjust]{cite}

\usepackage{amsmath} 
\usepackage[nolist]{acronym}
\usepackage{xcolor}
\usepackage{amsmath}
\usepackage{amssymb}
\usepackage{mathtools}
\usepackage{subcaption}
\usepackage{hyperref}
\hypersetup{
    colorlinks,
    citecolor=red,
    filecolor=black,
    linkcolor=blue,
    urlcolor=blue
}
\usepackage[ruled]{algorithm2e}
\usepackage{mathrsfs}
\usepackage[letterpaper, left=54pt, right=54pt, bottom=54pt, top=54pt]{geometry} 

\newcommand{\gh}{\hat{\mathcal{G}}}
\newcommand{\fs}{\mathcal{F}}
\newcommand{\ws}{\mathcal{W}}

\newcommand{\R}{\mathbb{R}}

\newcommand{\N}{\mathbb{N}}

\newcommand{\psd}{\mathbb{S}}

\newcommand{\gs}{\mathcal{G}}

\DeclarePairedDelimiter{\norm}{\lVert}{\rVert}

\DeclarePairedDelimiterX{\inp}[2]{\langle}{\rangle}{#1, #2}

\DeclareMathOperator*{\argmin}{arg\!\,min}

\DeclareMathOperator*{\find}{find}

\newtheorem{thm}{Theorem}[section]

\newtheorem{defn}{Definition}[section]

\newtheorem{rem}{Remark} 
\newtheorem{prob}{Problem} 
\newtheorem{assum}{Assumption} 
\usepackage{balance}
\usepackage{color}
\usepackage{ulem}

\title{\LARGE \bf Assessing the Quality of a Set of Basis Functions for Inverse \\ Optimal Control via Projection onto Global Minimizers}

\author{Filip Be\v{c}anovi\'{c}$^{1,4,*}$ Jared Miller$^{2,*}$ Vincent Bonnet$^3$ Kosta Jovanović$^4$ Samer Mohammed$^1$
\thanks{$^*$ Denotes equal contribution.}
\thanks{$^1$F. Be\v{c}anovi\'{c}, S. Mohammed are with the Université Paris-Est Créteil, Créteil, France. $^2$J. Miller is with the  Robust Systems Laboratory, ECE Department, Northeastern University, Boston, MA 02115. $^3$V. Bonnet is with the Laboratory of Analysis and Architecture of Systems, Toulouse. $^4$F. Be\v{c}anovi\'{c}, K. Jovanović are with the School of Electrical Engineering, University of Belgrade, Belgrade, Serbia.}
\thanks{© 2022 IEEE.  Personal use of this material is permitted.  Permission from IEEE must be obtained for all other uses, in any current or future media, including reprinting/republishing this material for advertising or promotional purposes, creating new collective works, for resale or redistribution to servers or lists, or reuse of any copyrighted component of this work in other works. This is the author's accepted manuscript of an article published in 61st IEEE Conference on Decision and Control. The final published version is available online at: \href{https://doi.org/10.1109/CDC51059.2022.9993342}{https://doi.org/10.1109/CDC51059.2022.9993342}.
}
\thanks{Cite as:
F. Bečanović, J. Miller, V. Bonnet, K. Jovanović and S. Mohammed, "Assessing the Quality of a Set of Basis Functions for Inverse Optimal Control via Projection onto Global Minimizers," 2022 IEEE 61st Conference on Decision and Control (CDC), Cancun, Mexico, 2022, pp. 7598-7605.
}
}

\begin{document}
\maketitle
\thispagestyle{empty}
\pagestyle{empty}
\begin{abstract}
\label{sec:abstract}
Inverse optimization (Inverse optimal control) is the task of imputing a cost function such that given test points (trajectories) are (nearly) optimal with respect to the discovered cost. Prior methods in inverse optimization assume that the true cost is a convex combination of a set of convex basis functions and that this basis is consistent with the test points. However, the consistency assumption is not always justified, as in many applications the principles by which the data is generated are not well understood. This work proposes using the distance between a test point and the set of global optima generated by the convex combinations of the convex basis functions as a measurement for the expressive quality of the basis with respect to the test point. A large minimal distance invalidates the set of basis functions. The concept of a set of global optima is introduced and its properties are explored in unconstrained and constrained settings. Upper and lower bounds for the minimum distance in the convex quadratic setting are implemented by bi-level gradient descent and an enriched linear matrix inequality respectively. Extensions to this framework include max-representable basis functions, nonconvex basis functions (local minima), and applying polynomial optimization techniques.
\end{abstract}
\section{Introduction}
\label{sec:introduction}

There is a consensus in the literature on biomechanics, human motor control, and robotics to say that numerous human motions can be modeled and predicted using optimal control approaches \cite{berret2011evidence, albrecht2012bilevel, mombaur2010human, lin2016human, panchea2018human, westermann2020inverse}. This is especially the case for repetitive motions or tasks that are subject to strict biomechanical constraints such as walking \cite{clever2016inverse}. To determine the cost function(s) used by humans, the use of \ac{IOC} methods has been proposed and extensively studied \cite{berret2011evidence, albrecht2012bilevel, mombaur2010human, lin2016human, panchea2018human, westermann2020inverse}. Let a general motion be represented as a vector $x$ from $\R^n$, and in particular, let observations be denoted by $y$. The application of \ac{IOC} requires at least one (nearly) optimal observation $y \in \R^n$, knowledge of the constraints related to the biomechanics of the task and of the human body that define the set of feasible motions $X \subseteq \R^n$, and  a finite-cardinality \textit{set of basis cost functions} $\fs = \{ f_i(x) \}_{i=1}^m$. Under the assumption that the cost function underlying human motion is a convex combination of functions from the basis $f(x) = \sum_{i=1}^m \alpha_i f_i(x)$, the goal of \ac{IOC} is to identify parameters $\alpha_i$ such that the observation(s) $y$ is (are) optimal. The basis functions are chosen based on physiological observations and are often related to the mechanical energy or motion smoothness \cite{berret2011evidence}.

\ac{IOC} is often solved in human motor control or robotics by using a computationally expensive bi-level optimization approach \cite{mombaur2010human, berret2011evidence, albrecht2012bilevel, clever2016inverse}. The bi-level approach consists of two optimization processes, one nested inside the other, that explicitly aim to minimize over the parameters $\alpha$ the Root-Mean-Square-Error between optimal observation $y$ and the result of the \textit{so-called} \ac{DOC}  problem that returns the optimal solution $x^*$ given the cost function $f(x)$.
This approach tends to be abandoned due to its long execution time, but has the merit of being able to deal with relatively \textit{noisy} input data \cite{aswani2018inverse} since the explicit minimization of the distance between $y$ and the result of the inner optimization problem is accounted for.

Recently, an approximate solution method for the \ac{IOC} problem was proposed based on a relaxation of the Karush-Kuhn-Tucker (KKT) optimality conditions, allowing to look for solutions by solving a linear least-square KKT-\textit{residual} minimization \cite{keshavarz2011imputing}. Even for high-dimensional non-linear models of the DOC, the \ac{IOC} admits an almost instantaneous solution, but requires the observation $y$ to be within, or very close, to the set of possible solutions of the \ac{DOC} denoted $\gs^c$ later (or $\gs$ in the absence of constraints).  This method has motivated a fair amount of work in robotics \cite{jin2019inverse, englert2017inverse, jin2020inverse, jin2021inverse} and human-motion analysis \cite{lin2016human, panchea2018human, westermann2020inverse}. The method is based on a relaxation of the exact feasibility problem of inverse optimization, formulated as Problem \ref{prob:feas_cons} (\textit{resp.} \ref{prob:feas_uncons}), for a constrained (\textit{resp.} unconstrained) direct optimization problem.

Though these approximate methods \cite{keshavarz2011imputing, johnson2013inverse, jin2019inverse} have considerable computational advantages with respect to the bi-level approach, they are known to be sensitive to noise \cite{panchea2017inverse, thai2018imputing} and are statistically inconsistent \cite{aswani2018inverse}. However, when dealing with human motion inherent modelling errors arise in treatment of the complex musculoskeletal system, with non negligible measurement errors. For example, joint angles, that are often used as optimal observation $y$ for \ac{IOC}, can have several degrees of error at the hip level during ground walking \cite{stagni2000effects}. 

Few studies \cite{colombel2021reliability} propose to evaluate how well a basis of cost functions $\fs$ can represent a given motion $y$. One way could be to perform repeated \ac{IOC}-\ac{DOC} cycles with different initialization for either the bi-level or the approximate approaches, and look for the smallest obtainable error. Another approach would be to use the residual of the approximate KKT approach \cite{keshavarz2011imputing}. Unfortunately, what constitutes a good value for the residual can wildly vary from problem to problem, and even within the same problem depending on the conditioning \cite{colombel2021reliability}. 

This work proposes to evaluate how well a basis of cost functions $\fs$ can represent a given motion $y$.
To make this concept intuitive, the concept of the \textit{set of global minima} $\gs^c$ ($\gs$) is introduced. The problem of evaluating a basis of cost functions $\fs$ is then formulated as the minimum-distance projection from the observation $y$ to the set of global minima $\gs^c$ ($\gs$), which is equivalent to finding the global minima of the general non-convex bi-level problem with no reliable way of finding global solutions. Under some assumptions, one can bound this distance from below. Meanwhile the distance between any point from $\gs^c$ to $y$ yields an upper bound on the minimum distance. One should strive to make the bound tighter by picking at least a local minimum of the bi-level problem. An approach based on an enhanced LMI will be presented for bounding the distance from below. As medium-sized LMIs are readily solved, this approach applies to problems of moderate sizes.
 
The set of global minima $\gs^c$ ($\gs)$ is characterized for convex bases $\fs$, both with constrained and unconstrained \ac{DOC} models. The problem of checking whether there exists a cost function parametrization $\alpha$ which makes the observation $y$ exactly optimal is formulated as a linear program. The projection problem, equivalent to the bi-level approach, is then re-formulated as a single-level optimization process using the classical KKT-transformation of bi-level problems \cite{dempe2015bilevel}. This reformulation allows for the application of enhanced LMIs for lower bounding the distance, in cases where the \ac{DOC} model is polynomial in cost and constraints.

Further considerations are restricted to unconstrained and constrained quadratic programs as \ac{DOC} models. Meaning the members of the set of basis functions $\fs$ are assumed quadratic and the constraints are, if existent, assumed linear. The geometry of the corresponding sets $\gs$ and $\gs^c$ are investigated, as well as the mappings between optimal motions $x$, cost function parametrizations $\alpha$, and the dual variables $\lambda$ and $\mu$ when constraints are present, through the extended global minima sets $\gh$ and $\gh^c$.
In this context, the contributions of this work are,
\begin{itemize}
    \item The formulation of the \ac{pgm} problem \ref{prob:pgm},
    \item The characterization of global optima set geometry,
    \item The use of numerical algorithms to determine upper and lower bounds in the convex quadratic setting.
\end{itemize}

This paper has the following structure: 
Section \ref{sec:prelim} defines notation and introduces the problems that will be addressed in this work.
Section \ref{sec:geometry_uncons} describes the structure of unconstrained global-optima sets, with specific reference to cases where each $f_j$ is a strongly convex quadratic. Section \ref{sec:geometry_cons} analyzes constrained global-optima sets, and similarly focuses on the \ac{QP} case where the objective functions are weakly convex.
Section \ref{sec:numerical_methods} acquires upper bounds for program \ref{prob:pgm} using local search and lower bounds  through an \ac{LMI}. Section \ref{sec:numerical_methods} presents numerical examples of the proposed \ac{pgm} approach. Section \ref{sec:extensions} details extensions to the current method, including adding support for maximum-representable functions in the cost functions set  $\fs$. 
\begin{acronym}[ProjGM]
\acro{DOC}{Direct Optimal Control}
\acro{IOC}{Inverse Optimal Control}
\acro{PIO}{Parametric Inverse Optimization}
\acro{KKT}{Karush-Kuhn-Tucker}
\acro{LMI}{Linear Matrix Inequality}
\acroplural{LMI}[LMIs]{Linear Matrix Inequalities}
\acroindefinite{LMI}{an}{a}

\acro{LP}{Linear Program}
\acro{POP}{Polynomial Optimization Problem}
\acro{pgm}[ProjGM]{Projection onto Global Minimizers}
\acro{PD}{Positive Definite}
\acro{PSD}{Positive Semidefinite}
\acro{QCQP}{Quadratically Constrained Quadratic Programming}
\acro{QP}{Quadratic Program}
\acro{SDP}{Semidefinite Program}
\acro{SOS}{Sum of Squares}
\end{acronym}
\section{Preliminaries}
\label{sec:prelim}
\subsection{Notations}
The set of real numbers is $\R$, the $n$-dimensional Euclidean space is $\R^n$, and its nonnegative real orthant is $\R^n_+$. The vector of all zeros is $0$ and of all ones is $\mathbf{1}$. The set of natural numbers (incl. 0) is $\N$, and its subset between $1$ and $N$ is $1..N$. The $n$-dimensional probability simplex is $\Delta^n$.  The set of $m\times n$ real-valued matrices is $\R^{m \times n}$. The set of $n\times n$ symmetric real matrices satisfying $Q = Q^T$ is $\psd^n$.
A symmetric matrix is \ac{PSD} $(Q \in \psd^n_+)$ if $\forall x \in \R^n: \ x^T Q x \geq 0$, and is \ac{PD} $(Q \in \psd^n_{++})$ if $\forall x \in \R^n, x \neq 0: x^T Q x > 0$.

Let $X$ and $Y$ be a pair of spaces. A single valued function may be written as $f:X \rightarrow Y$, and a set valued function $F: 2^X \rightarrow 2^Y$ may be expressed as $F: X \rightrightarrows Y$.
The projection $\pi^x: (x, y) \mapsto x$ operation applied to a set $\hat{\mathbb{J}} \subseteq X \times Y$ is, 
\begin{equation}
    \mathbb{J}(x) = \pi^x \hat{\mathbb{J}}(x,y) = \{x \mid \exists y \in Y: (x, y) \in \hat{\mathbb{J}}\}.
\end{equation}

\subsection{Problem Statement}
A convex combination cost $f_\alpha$ of basis functions from $\fs = \{ f_i(x) \}_{i=1}^m$ given weights $\alpha$ may be expressed as,
\begin{align}
    f_\alpha(x) &= \textstyle \sum_{j=1}^m \alpha_j f_j(x) & \alpha \in \Delta^m \label{eq:cost_construction}
\end{align}
where $\Delta^m$ is the $m$-dimensional probability simplices $\Delta^m = \{\alpha \in \R^m \mid \alpha \geq 0, \sum_i \alpha_i = 1\}.$
This paper attempts to answer the following two problems:

\begin{prob}[Feasibility]
\label{prob:feas}
Does there exist an $\alpha \in \Delta^m$ such that
 $y$ is a global optimum?
\begin{align}
    \find_{\alpha \in \Delta^m} & \ y \in \textstyle \argmin_{x \in X} \sum_{j=1} \alpha_j f_j(x)
\end{align}
\end{prob}

\begin{prob}[\ac{pgm}]
\label{prob:pgm}
What is the distance from $y$ to the set of global optima?
\begin{subequations}
\label{eq:pgm}
\begin{align}
    p^* = \min_{x \in X} & \ \norm{y-x}_2^2 \\ 
    \exists \alpha \in \Delta^m & \mid x \in \textstyle \argmin_{x' \in X} \sum_{j=1} \alpha_j f_j(x')
\end{align}
\end{subequations}
\end{prob}

If $\alpha$ solves the feasibility program \ref{prob:feas}, then the objective from problem \ref{prob:pgm} will necessarily be $p^* = 0$ with $x = y$.
Both problems will be treated under the following assumption (the same setting as used in \cite{keshavarz2011imputing}),
\begin{assum}[Convexity]
\label{assum:convex_weak}
Each $f_j$ is a $C^1$ convex function and $X$ is a convex set.
\end{assum}

Problem \ref{prob:feas} under Assumption \ref{assum:convex_weak} may be decided by analyzing the feasibility of a simple $n$-dimensional \ac{LP} in $\alpha$. 
Problem \ref{prob:pgm} is a bilevel optimization problem that assesses the quality of the basis $\fs$ from \eqref{eq:cost_construction} in characterizing $y$ as an optimal point. A large distance $p^*$ from Program \eqref{eq:pgm} indicates that the  basis $\fs$ does not accurately describe $y$ as a global minimum and therefore $\fs$ should be redesigned.
Problem \ref{prob:pgm} is a generically nonconvex and nontrivial problem even under Assumption \ref{assum:convex_weak}.
Local search and trust-region methods may be used to find upper bounds for $p^*$, but exact computation of $p^*$ is typically intractable. \Iac{LMI} is introduced to obtain lower bounds for $p^*$, and this \ac{LMI} relaxation is tight to $p^*$ if the \ac{PSD} matrix variable is rank-1.
\begin{figure}[h]
    \centering
    \includegraphics[width=0.75\linewidth]{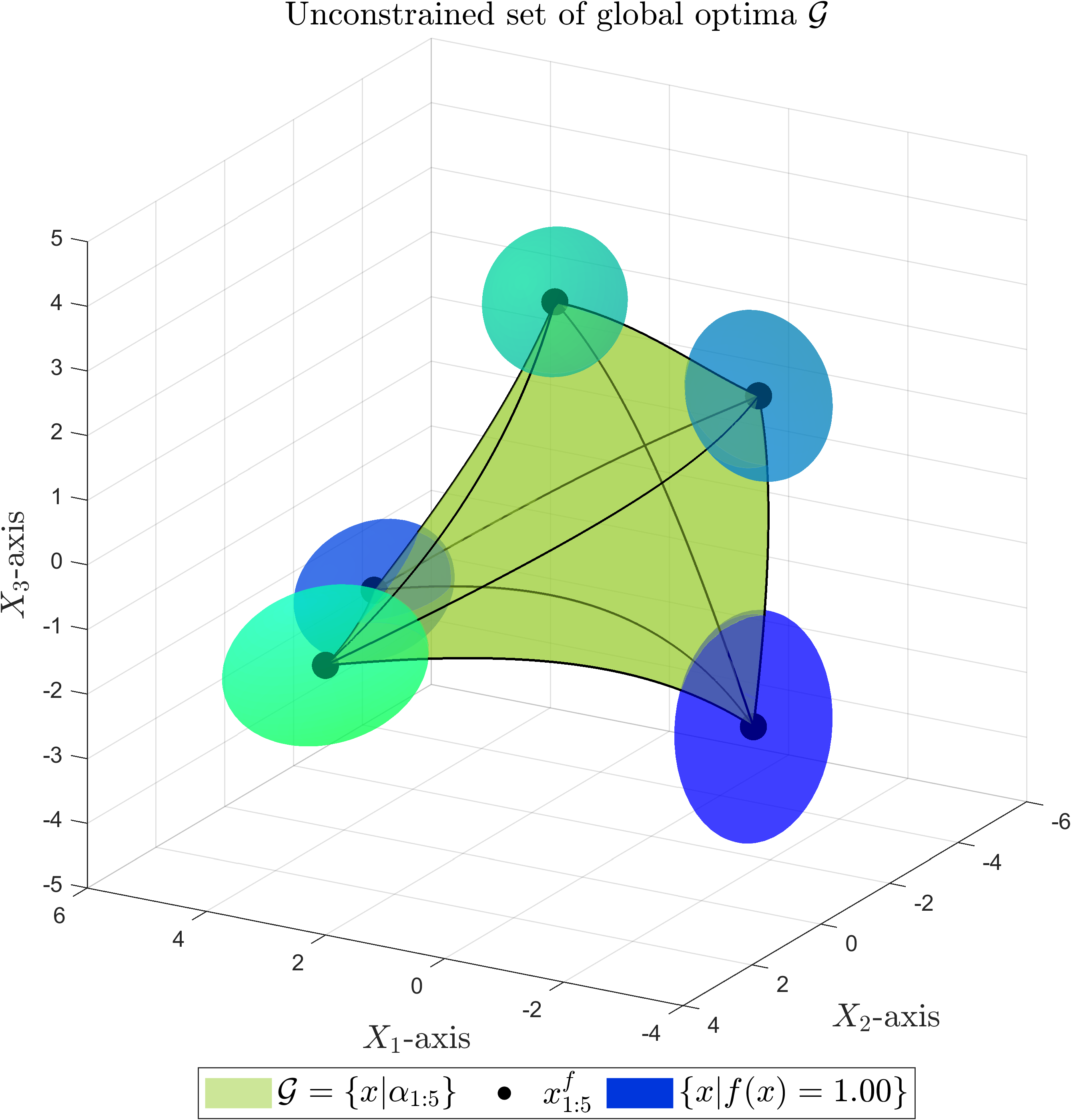}
    \caption{A set of unconstrained global optima in 3D}
    \label{fig:uncons3d}
\end{figure}

\section{Geometry of Unconstrained Optimal Sets}
\label{sec:geometry_uncons}

This section formulates and defines properties of sets of global optima when $X = \R^n$. Due to repeated usage throughout this section, the symbol $\ws$ will denote $\R^n \times \Delta^m$.

\subsection{Convex Case}
The set of global optima and weights associated with a $\fs$ yielding mixed cost functions \eqref{eq:cost_construction} given Assumption  \ref{assum:convex_weak} is,
\begin{align}
    \gh &= \{(x, \alpha) \in \ws \mid \nabla_x f_\alpha(x) = 0\}.
\end{align}

Convexity of $f_\alpha$ implies that every local minimizer with $\nabla_x f_\alpha(x)=0$ is also a global minimizer.

As condition $\nabla_x f_\alpha(y) = 0$ is linear in $\alpha$, the feasibility problem \ref{prob:feas} given $y \in \R^n$ may be posed as an \ac{LP} in $\alpha$, 
\begin{prob}[Unconstrained Feasibility]
\label{prob:feas_uncons}
\begin{align}
\label{eq:feas_uncons_lp}
    \textstyle \find_{\alpha \in \Delta^m} \ (y, \alpha) \in \gh.
\end{align}
\end{prob}

The unconstrained  \ac{pgm} problem for the same $y$ is,
\begin{prob}[Unconstrained \ac{pgm}]
\label{prob:pgm_uncons}
\begin{align}
    p^* &= \min_{(x, \alpha) \in \gh} \norm{y-x}_2^2. \label{eq:pgm_uncons}
\end{align}
\end{prob}

\subsection{Quadratics Only}
Further analysis in this section will require the following assumption,
\begin{assum}[Strictly Convex Quadratics]
\label{assum:quad_uncons}
Functions in $\fs$ are strictly convex quadratics with,
\begin{subequations}
\label{eq:quad_uncons}
\begin{align}
    f_j(x) &=  (x-x_j^f)^T Q_j (x-x_j^f)/2 & \forall j = 1..m \\
    \nabla_x f_j(x) &=  Q_j (x-x_j^f) & \forall j = 1..m.
\end{align}
\end{subequations}
where all elements in $(Q_j, x_j^f)$ are bounded and $Q_j \in \psd_{++}^n$ for all $j=1..m$.
\end{assum}

The optima-weight sets $\gh(x, \alpha)$ and $\gs(x)$ under Assumption \ref{assum:quad_uncons} is,
\begin{align}
\label{eq:gh_quad_uncons}
    \gh &= \textstyle \{(x, \alpha) \in \ws \mid \left(\sum_j \alpha_j Q_j \right)x= \sum_j \alpha_j Q_j x_j^f\} \\
    \gs &= \pi^x \gh.
\end{align}
\begin{rem}
Letting $\{e_j\}_{j=1}^m$ be the standard basis vectors in $\R^m$, the points $\{(x_j^f, e_j)\}_{j=1}^m$ are all members of $\gh$.
\end{rem}
\begin{rem}
All descriptor constraints in $\gh$ from \eqref{eq:gh_quad_uncons} have polynomial degree at most two, and there is a bilinearity between $(\alpha, x)$.
\end{rem}

The mapping $\kappa: \Delta^m \rightarrow \gs$ with  
$\kappa(\alpha) = x^*_\alpha$ in terms of finding an optimal $x$ minimizing $f_\alpha$ is single-valued, and has an expression,
\begin{equation}
\label{eq:opt_x_uncons}
\kappa(\alpha) = \textstyle x^*_\alpha = \left(\sum_j \alpha_j Q_j \right)^{-1}\left( \sum_j \alpha_j Q_j x_j^f\right).
\end{equation}
\begin{rem}
Conversely, the mapping $\kappa^{-1}: \gs \rightrightarrows \Delta^m$ is set-valued. 
Its values are the compact polytopic sets containing feasible points of program \eqref{eq:feas_uncons_lp}.
\end{rem}
\begin{thm}
\label{thm:cont_onto}
The map $\kappa(\alpha)$ from \eqref{eq:opt_x_uncons} is a continuous surjection from $\Delta^m$ onto $\gs$ under Assumption \ref{assum:quad_uncons}.
\end{thm}
\begin{proof}
Surjection holds by definition of $\gh$ in \eqref{eq:gh_quad_uncons}: $x$ is only a member of $\gh$ if there exists an $\alpha'\in\Delta^m$ such that $x = x^*_{\alpha'}$. Continuity of $x_\alpha^*$ is based on continuity of the matrix inverse $A^{-1}$ for all nonsingular matrices $A \in \R^{n\times n}$.
\end{proof}

\begin{defn}[Compact] A set $X \in \R^n$ is \textit{compact} if it is closed and bounded (Heine-Borel). A consequence is that $X$ is compact if there exists a finite $R > 0$ such that $X$ is a subset of the ball with radius $R: \ \{x \mid \norm{x}_2 < R\}$ \cite{munkres2000topology}.
\end{defn}
\begin{thm}
\label{thm:compact}
Under Assumption \ref{assum:quad_uncons}, the set $\gh$ is compact.
\end{thm}
\begin{proof}
There exists a finite quantity $R > 0$ that satisfies $\max_j \norm{x^f_j}_2 < R$ due to boundedness of $x^f$. It is implied that $\max_{\alpha \in \Delta^m} \norm{\sum_{j=1}^m \alpha_j x^f_j}_2 < R$ by convexity of the norm $\norm{\cdot}_2$. Let $\Lambda$ be the solution to,
\begin{equation}
    \Lambda = \min_{\alpha \in \Delta^m} \lambda_{min}\left( \textstyle \sum_j \alpha_j Q_j\right).
\end{equation}
It holds that $\Lambda > 0$ because \ac{PD} matrices form a convex (non-pointed) cone. The maximum norm of any global-optimal point in $\gs$ is bounded above by,
\begin{subequations}
\begin{align}
    \norm{\kappa(\alpha)}_2 &= \textstyle \norm{\left(\sum_j \alpha_j Q_j \right)^{-1}\left( \sum_j \alpha_j Q_j x_j^f\right)}_2 \\
    &\textstyle \leq (1/\Lambda) \norm{\left( \sum_j \alpha_j Q_j x_j^f\right)}_2 \leq R / \Lambda < \infty.
\end{align}
\end{subequations}

The compact set $\{(x, \alpha) \in \ws \mid \norm{x}_2 \leq R/\Lambda\}$ is a superset of $\gh$ from \eqref{eq:gh_quad_uncons}, which proves that $\gh$ is compact.
\end{proof}

\begin{defn}[Path-Connected]A set $X$ is path-connected if for every two points $x^0, x^1 \in X$ there exists a continuous path (curve) $\omega: [0, 1] \rightarrow X$ with $\omega(0) = x^0, \ \omega(1) = x^1$ such that $\omega(t) \in X \ \forall t \in [0, 1]$ \cite{munkres2000topology}.
\end{defn}
\begin{thm}
\label{thm:path_quad}
Under Assumption \ref{assum:quad_uncons}, the set $\gs$ is path-connected.
\end{thm}
\begin{proof}
Let $x^0, x^1$ be a pair of distinct points in $\gs$. Choose $\alpha^0 \in \kappa^{-1}(x^0)$ and $\alpha^1 \in \kappa^{-1}(x^1)$ as weights generating the optimal points $x^0, x^1$. 
A path $\omega:[0, 1]\rightarrow \Delta^m$ may be drawn between the points by $\omega(t): \alpha^0 t + \alpha^1(1-t)$. The path $(\omega(t), \kappa(\omega(t)) )$ remains inside $\gh$ for all $t \in [0, 1]$ by continuity of $\kappa$ from Theorem \ref{thm:cont_onto}. This containment holds for all pairs $(x^0, x^1)$, so $\gs$ is path-connected.
\end{proof}
\begin{rem}
No conclusions can be drawn in this manner about path-connectedness of $\gh$.
\end{rem}

Figure \ref{fig:uncons3d} depicts the set of unconstrained global minima $\gs$ for $x \in \R^3$, generated by $5$ different basis functions, thus $\alpha \in \Delta^5$. Optimal points of the individual basis functions, as well as the $1$-level-sets of the quadratics are plotted, in order to give a sense of pairs $(Q_j, x^f_j)_{j=1}^5$ involved. The black edges connecting pairs of $x^f_j$ correspond to edges of the geometric shape, but also to sets of global minima of pairwise function combinations. Compactness and path-connectedness are visually obvious for this example.
\section{Geometry of Constrained Optimal Sets} \label{sec:geometry_cons}
This section will extend \ref{sec:geometry_uncons} to consider the case when $X \subsetneq \R^n$ is a convex constraint set. 

\subsection{Convex Case}
The following assumption and representation is required,
\begin{assum}
\label{assum:con_rep}
There exists matrices $A_{eq} \in \R^{q \times n}, \ b_{eq} \in \R^{q}$ and $C^1$ convex functions $\{g_k(x)\}_{k=1}^r$ such that,
\begin{equation}
    X = \{x \in \R^n \mid A_{eq} x = b_{eq}, \ g_k(x) \leq 0 \  \forall k = 1..r\}.
\end{equation}
\end{assum}
\begin{assum}[Slater's Condition]
\label{assum:slater}
There exists a point $x'\in \R^n$ such that $A_{eq} x' = b$ and $g_k(x') < 0 \ \forall k = 1..n$, meaning $X$ is non-empty.
\end{assum}

Define $\mu \in \R_+^r$ and $\lambda \in \R^q$ as dual variables against the inequality and equality constraints describing $X$ respectively. The \ac{KKT} necessary conditions are sufficient to classify all optimal (minimizer) points of $f_\alpha(x)$ given a weighting $\alpha \in \Delta^n$ and assumptions \ref{assum:convex_weak}, \ref{assum:con_rep}, \ref{assum:slater} \cite{boyd2004convex}:
\begin{subequations}
\label{eq:kkt_cons}
\begin{align}
    &\textstyle \nabla_x f_\alpha(x) + A_{eq}^T \lambda + \sum_{k=1}^r \mu_k \nabla_x g_k(x) = 0\\
    & A_{eq} x = b \\
    & g_k(x) \leq 0, \ \mu_k \geq 0 \qquad \forall k = 1..n\\
    & \textstyle \sum_{k} \mu_k g_k(x) = 0.
\end{align}
\end{subequations}

Define the symbol $\ws^c = \R^n \times \Delta^n \times \R_+^r \times \R^q$ as the resident set containing $(x, \alpha, \mu, \lambda)$.

The optima-weight set in the constrained case is,
\begin{align}
\label{eq:gh_cons}
    \gh^c &= \{(x, \alpha, \mu, \lambda) \in \hat{\ws} \mid \text{\ac{KKT} conditions \eqref{eq:kkt_cons} hold}\} \\
    \gs^c &= \pi^x \gh^c.
\end{align}

The feasibility \ac{LP} to check if a $y \in \R^n$ is constrained-optimal (similar to \eqref{eq:feas_uncons_lp} for the unconstrained case) is,
\begin{prob}[Constrained Feasibility]
\label{prob:feas_cons}
\begin{align}
\label{eq:feas_cons_lp}
    \textstyle \find_{\alpha \in \Delta^m, \ \mu \in \R_+^r, \ \lambda \in  \R^q} \ (y, \alpha, \mu, \lambda) \in \gh^c,
\end{align}
\end{prob}
with a constrained \ac{pgm} program,
\begin{prob}[Constrained \ac{pgm}]
\label{prob:pgm_cons}
\begin{align}
    p^* = \min_{(x, \alpha, \mu, \lambda) \in \gh^c} \norm{y-x}_2^2. \label{eq:pgm_cons}
\end{align}
\end{prob}

\subsection{Quadratic Programming} \label{sec:qp}
We note that the special case of \ac{QP} involves candidate functions and a constraint set,
\begin{subequations}
\label{eq:qp_param}
\begin{align}
    f_j &= x^T Q_j x/2 + {\varphi_j}^T x, \ \forall j =1..m \label{eq:qp_objective} \\
    X &= \{x \in \R^n \mid A_{eq} x = b_{eq}, \ A x \leq b\}, \label{eq:qp_X}
\end{align}
\end{subequations}for matrices $\{Q_j \in \psd_+^{n}, \ \varphi_j \in \R^n\}_{j=1..m}$ and $A_{eq} \in \R^{q \times n}, \  b_{eq} \in \R^q, \ A\in \R^{r \times n}, \ b \in \R^r$ such that Assumption \ref{assum:slater} (Slater) holds.

The constrained-optimal solution map given $\alpha$ and the parameters in \eqref{eq:qp_param} is,
\begin{equation}
\label{eq:qp_argmin}
    \kappa^c(\alpha) = \textstyle \argmin_{x \in X} \sum_{j=1}^m \alpha_j f_j(x).
\end{equation}

\begin{rem}
Due to possible weak convexity of some cost functions in $\fs$, there may exist points $\alpha \in \Delta^m$ such that $\kappa^c(\alpha)$ is set-valued rather than single-valued. This will occur when $Q_\alpha = (\sum_{j=1}^m \alpha_j Q_j)$ is rank-deficient and $(\sum_{j=1}^m \alpha_j \varphi_j)$ is orthogonal to $Q_\alpha$'s nullspace.
\end{rem}

In this case, the possibly discontinuous selection (translation of the minimum map),
\begin{equation}
\label{eq:y_selection}
    s(\kappa^c(\alpha)) = \argmin_{x \in \kappa^c(\alpha)} \norm{y-x}_2^2,
\end{equation} will denote a constrained-optimal point in $\kappa^c(\alpha)$  that is closest to $y$. Finding a selection $s(\kappa^c(\alpha))$ requires solving a second \ac{QP} over the $X$-intersected subspace of solutions of \eqref{eq:qp_argmin}. This two-step approach involving a minimal selection was also performed in \cite{spjotvold2007continuous}.

Figure \ref{fig:set_containment} depicts constrained global optima $\gs^c$, and it's unconstrained version $\gs$ generated from the same basis function set $\fs$. 

Here $2$ variables, $3$ cost functions, and $4$ inequality constraints with $(x, \alpha, \mu) \in \R^2 \times \Delta^3 \times \R^4_+$ are considered.
Unconstrained and constrained optimal points of the individual basis functions, $x^f_{1:5}$ and $x^{f, c}_{1:5}$, are shown as colored dots. The black square denotes the feasible region. The blue ellipses are the level sets of the quadratic cost functions, chosen to highlight the spots where they are tangential to the constrained set boundary (implying the position of their constrained minima). The red box highlights the spot where the intersection of $\gs^c$ and the complement of $\gs$ is non-empty. 

\begin{figure}[h]
    \centering
    \includegraphics[width=0.75\linewidth]{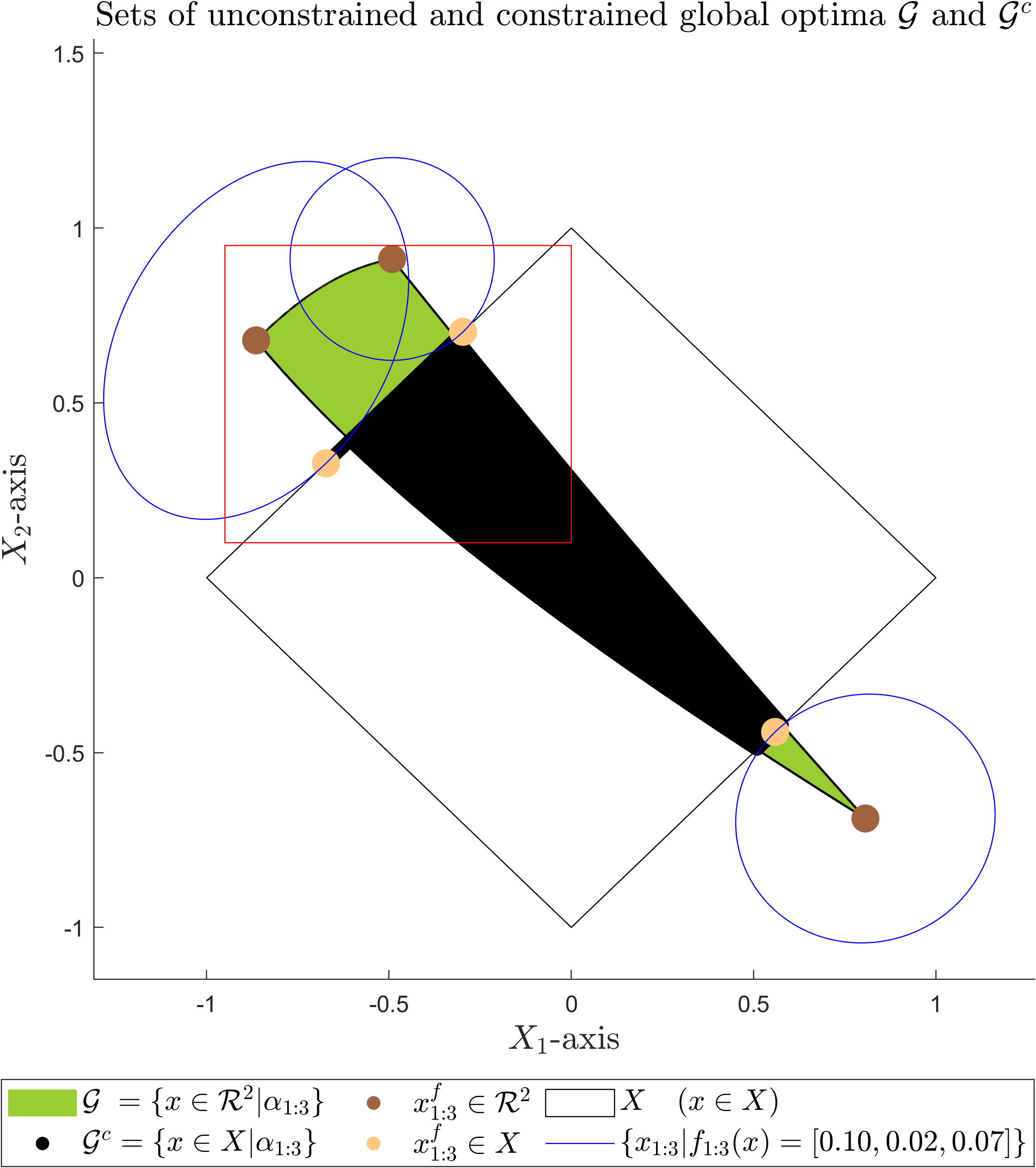}
    \caption{The set $\gs^c$ is not necessarily a subset of $\gs$}
    \label{fig:set_containment}
\end{figure}
\section{Numerical Methods} \label{sec:numerical_methods}

This section will present numerical approaches to find upper and lower bounds for the Unconstrained \ac{pgm} \ref{prob:pgm_uncons} and Constrained \ac{pgm} \ref{prob:pgm_cons}.

\subsection{Upper Bounds}
\label{sec:upper_bounds}

Upper bounds of $p^*$ may be computed through sampling and local search optimization. The \ac{pgm} problems may be formulated solely in terms of $\alpha$ through the use of the optimization maps $\kappa(\alpha)$ and $\kappa^c(\alpha)$ for $\alpha \in \Delta^*$ (bi-level optimization).

The objective in \ref{prob:pgm_uncons} and its gradient with respect to $\alpha$ is,

\begin{align}
    F(\alpha) &= \norm{\kappa(\alpha) - y}_2^2 \label{eq:cost_uncons}\\
    \nabla_\alpha F(\alpha)_j &= \textstyle2 (\kappa(\alpha) - y)^T \left( \sum_j \alpha_j Q_j\right)^{-1} Q_j (x_j^f - \kappa(\alpha)).\nonumber
\end{align}

A Hessian for $F$ may be similarly derived in closed form (omitted due to space constraints). Local solvers such as \texttt{fmincon} in MATLAB or Manopt\footnote{As $\Delta^m$ is not a manifold, the Hadamard parameterization $\Delta^m = \{z \odot z \mid \norm{z}_2^2 = 1\}$ may be employed \cite{mckenzie2021hadamard}.} \cite{manopt} may be given this value and derivative information to perform optimization over $\alpha \in \Delta^m$ (e.g. L-BFGS, trust-region). 

The constrained objective in \ref{prob:pgm_cons} may be expressed as,
\begin{equation}
\label{eq:cost_cons}
    F^c(\alpha) = \norm{s(\kappa^c(\alpha))-y}_2^2.
\end{equation}

The constrained objective \eqref{eq:cost_cons} is not generally differentiable with only the weak convexity assumption. Two possible options to minimize $F^c(\alpha)$ include gridding $\Delta^m$, and using \texttt{fmincon} in terms of $(x,\alpha, \mu, \lambda) \in \gh^c$ directly on the constrained optimization problem \ref{prob:pgm_cons}.

\subsection{Lower Bounds}
\label{sec:lower_bounds}
Lower bounds $p_{low} \leq p^*$ to problems \ref{prob:pgm_uncons} and \ref{prob:pgm_cons} in the quadratic case may be acquired through semidefinite programming. The presented method in this subsection is a \ac{QCQP} that is equivalent to the degree-1 moment-\ac{SOS} hierarchy \ac{LMI} enriched by additional constraints. 

\subsubsection{Unconstrained Lower Bound}
Every point $(x, \alpha) \in \ws$ defines a rank-1 \ac{PSD} matrix $M = [1 \ x \ \alpha] [1 \ x \ \alpha]^T$. Setting $M \in \psd^{1+n+m}_+ $ indexed by $(1, x, \alpha)$ as a matrix variable,
\begin{equation}
\label{eq:sdp_var_uncons}
    M = \begin{bmatrix}M_{11} &M_{1x} & M_{1\alpha} \\
    M_{x1} & M_{xx} & M_{x \alpha} \\
    M_{\alpha 1} &M_{\alpha x} & M_{\alpha \alpha}\end{bmatrix},
\end{equation}
the objective $\norm{y-x}_2^2$ may be converted into an affine expression, 
\begin{equation}
\label{eq:obj_affine}
    \textstyle \sum_{i=1}^n \left(M_{x_i x_i} - 2 y_i M_{1 x_i}\right) + \norm{y}^2_2
\end{equation}

Containment in $\gh$ from \eqref{eq:gh_quad_uncons} may be expressed as,
\begin{subequations}
\label{eq:uncons_grad}
\begin{align}
    & \textstyle \sum_{j=1}^m M_{\alpha_j 1} = 1 \label{eq:alpha_sdp_top}\\
    & M_{\alpha_j 1} \geq 0 \qquad \forall j = 1..m \label{eq:alpha_sdp_bot}\\
    & \textstyle \sum_{j=1}^m Q_j M_{x \alpha_j} - (Q_j x_j^f) M_{1 \alpha_j} = 0.
&\end{align}
\end{subequations}

Valid constraints are a set of relations that are always satisfied by an $M$ generated by the optimal point $(x^*, \alpha^*)$ solving Problem \ref{prob:pgm_uncons}. These valid constraints include,
\begin{subequations}
\label{eq:valid_uncons}
\begin{align}
    M_{\alpha_i 1} &= \textstyle \sum_{j=1}^m M_{\alpha_i \alpha_j} & & \forall i \neq j\label{eq:valid_uncons_top}\\
    M_{\alpha_i \alpha_j} &\geq 0 & & \forall i \neq j \label{eq:valid_uncons_bot}\\
    M_{\alpha_i \alpha_j} &\leq \alpha_i, \ M_{\alpha_i \alpha_j} \leq \alpha_j & & \forall i \neq j \label{eq:valid_uncons_simp0} \\
    M_{\alpha_i \alpha_i} &\leq M_{1 \alpha_i} & &  \forall i = 1..m \label{eq:valid_uncons_simp}\\
    M_{\alpha_i \alpha_j}
    &\leq 1/4 & & \forall i \neq j. \label{eq:valid_uncons_sym}
\end{align}
\end{subequations}

Constraints \eqref{eq:valid_uncons_top}-\eqref{eq:valid_uncons_bot} arise from multiplying together defining constraints for the simplex. The diagonal entry $M_{\alpha_i \alpha_i}$ will automatically be positive by  $M \in  \psd^{1+n+m}_+$. Constraints \eqref{eq:valid_uncons_simp0}-\eqref{eq:valid_uncons_sym} originate from observations about the simplex $\Delta^m$. Because $\alpha \in \Delta^m \subsetneq [0, 1]^m$, coordinate-wise multiplication will satisfy $\alpha_i \alpha_j \leq \alpha_i$ for all $i, j \in 1..n$. 

Constraint \eqref{eq:valid_uncons_sym} results from the fact that 
elementary symmetric polynomials on the probability simplex $\Delta^{\sigma}$ are maximized at the vector $\mathbf{1}/\sigma$. The elementary symmetric polynomial applied to $\alpha' \in \Delta^{\sigma}$ is $e_2^{\sigma}(\alpha') = \sum_{1 \leq i < j \leq n} \alpha'_i \alpha'_j$, and admits the maximum value of $\max e_2^{\sigma} = \binom{\sigma}{2}\frac{1}{\sigma^2}$. The valid inequality derived from  $\sigma=2$ with $\max e_2^{\sigma} = \frac{1}{4}$ is written in \eqref{eq:valid_uncons_sym}. Valid inequalities with higher $\sigma$ may be written at the cost of including a combinatorially increasing number of constraints.

\SetKwInOut{Input}{Output}
\begin{algorithm}[h]
\caption{Unconstrained \ac{LMI}}\label{alg:uncons_lmi}
\KwIn{$y, \ Q, \ x_f$}
\KwOut{ $p^*_{low}, \ M$ (or Infeasibility)}
Solve (or find infeasibility certificate):
\begin{subequations}
\label{eq:uncons_sdp}
\begin{align}
p^*_{low} = & \min_{M} \quad  \text{Objective \eqref{eq:obj_affine}}\\
& \text{Optimality \eqref{eq:uncons_grad},  Valid \eqref{eq:valid_uncons}} \\
& M_{11} = 1, \ M \in \psd^{1+n+m}_+.
\end{align}
\end{subequations}
\end{algorithm}

A rank-1 matrix solution $M$ of Algorithm \ref{alg:uncons_lmi} certifies that $p^*_{low} = p^*$. The optimal entries $(x, \ \alpha)$ can then be read from the solution's entries $(M_{1x}, \ M_{1 \alpha})$. Adding valid inequalities \eqref{eq:valid_uncons} can encourage rank-1 solutions of \ac{LMI} lower bound problem, refer to \cite{ahmadi2021semidefinite} for further examples of this phenomenon.

\subsubsection{Constrained Lower Bounds}
The lower bound SDP for the \ac{QP} setting in Section \ref{sec:qp} requires a matrix $M \in \psd_+^{1+n+m+r}$ and a vector $\lambda \in \R^{q}$ such that the entries of $M$ are indexed by $[1, x, \alpha, \mu]$. The equality multipliers $\lambda$ may be ommited from $M$ because there is no multiplication in \eqref{eq:kkt_cons} between terms that contain $x$ and $\lambda$.
The affine constraint interpretation of the \ac{KKT} conditions in \eqref{eq:kkt_cons} is,
\begin{subequations}
\label{eq:kkt_sdp}
\begin{align}
    & \textstyle \sum_{j=1}^m (Q_j M_{x \alpha_j} + \varphi_j M_{\alpha_j 1}) + A^T_{eq} \lambda + A^T M_{\mu 1}= 0 \\
    & A_{eq} M_{x 1} = b_{eq} \\
    & A M_{x 1} \leq b  \quad M_{\mu 1} \geq 0\\
    & -b^T M_{\mu 1} + \textstyle \sum_{k=1}^r A_{k} M_{x \mu_k} = 0.
    \label{eq:kkt_qp_slackness}
\end{align}
\end{subequations}

The notation $A_k$ in the complementary slackness constraint \eqref{eq:kkt_qp_slackness} indicates row $k$ of the matrix $A$.
Valid inequalities for the constrained \ac{QP} case include (with inequality constraint indices $(k, \ell)$),
\begin{subequations}
\label{eq:valid_qp}
\begin{align}
    M_{\mu_k \mu_\ell} &\geq 0 & & \forall k \neq \ell \\
    M_{\mu_k \alpha_i} &\geq 0 & & \forall k=1..r, \ i=1..m. \\
    M_{\mu_k \alpha_i} &\leq M_{\mu_k 1} & & \forall k=1..r, \ i=1..m.
\end{align}
\end{subequations}

Both $\alpha \in \Delta^m$ and $\mu \in \R_+^r$ are nonnegative, so their multiplications should also be nonnegative. The \ac{QP}-lower-bounding \ac{LMI} is,

\begin{algorithm}[h]
\caption{Constrained \ac{LMI}}\label{alg:cons_lmi}
\KwIn{$y, \ Q, \ \varphi, \ A, \ b, \ A_{eq}, \ b_{eq}$}
\KwOut{ $p^*_{low}, \ M$ (or Infeasibility)}
Solve (or find infeasibility certificate):
\begin{subequations}
\label{eq:cons_sdp}
\begin{align}
p^*_{low} = & \min_{\lambda \in \R^q, \ M} \quad  \text{Objective \eqref{eq:obj_affine}}\\
& \text{Simplex \eqref{eq:alpha_sdp_top}-\eqref{eq:alpha_sdp_bot}, \ KKT \eqref{eq:kkt_sdp}}   \\
& \text{Valid \eqref{eq:valid_uncons}, \eqref{eq:valid_qp}} \\
& M_{11} = 1, \ M \in \psd^{1+n+m+r}_+
\end{align}
\end{subequations}
\end{algorithm}
The accuracy of $p^*_{low} \leq p^*$ from Algorithm \ref{alg:cons_lmi} may be improved if an upper bound for $\mu_k^2$ $(M_{\mu_k \mu_k})$ was known for each $k$. Bisection-based approaches with convex cost heuristics may also be applied to the presented \acp{LMI} \cite{mohan2012iterative}.
\section{Numerical Examples} \label{sec:examples}
This section contains small-scale toy examples with visuals, to develop intuition. The examples will showcase flaws in the approximate projection method \cite{keshavarz2011imputing} when an inconsistent basis is supposed, and the robustness of the bi-level method. The proposed distance-bounding methods will also be displayed. Future works will showcase the method for larger-scale examples and will be applied to real data.

Assume that an observed human decision (or trajectory) can be represented by a vector $y \in \R^2$, or by a vector $y^c \in X \subset R^3$, and that it is generated via an unknown process (as is actually the case in real-world applications). Assume bases $\fs$ and $\fs^c$ contain $5$ and $3$ quadratic cost functions respectively, believed to be underlying the decision-making (or trajectory-generation). Figure \ref{fig:proj_to_uncons} depicts the $2$-dimensional test point $y \in \R^2$, and the set of unconstrained global optima $\gs$ which is generated using $\fs$. Figure \ref{fig:proj_to_cons} depicts the $3$-dimensional test point $y^c \in \R^3$, the feasible set $X$, and the set of constrained global optima $\gs^c$ generated by using $\fs^c$.

\begin{figure}[h]
    \centering
    \includegraphics[width=0.75\linewidth]{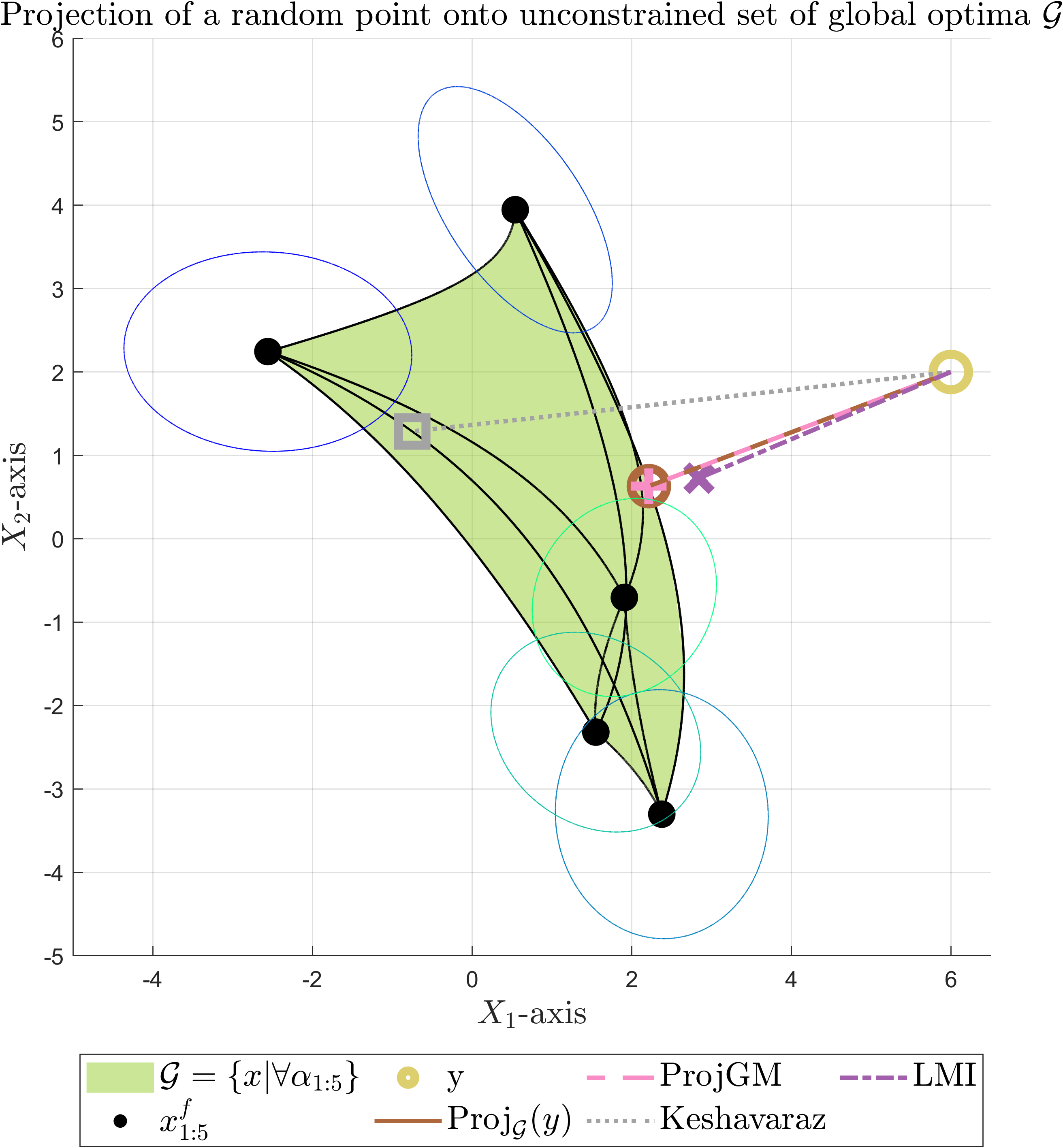}
    \caption{Comparison of the projection onto an unconstrained set of optima $\gs$.}
    \label{fig:proj_to_uncons}
\end{figure}

Both on Figure \ref{fig:proj_to_uncons} and \ref{fig:proj_to_cons} the actual minimum distance projection $\textrm{Proj}_{\gs}(y)$ is shown as reference, and is computed by a combined brute-force grid search and local bi-level search on the cost function parametrization $\alpha \in \Delta^5$ and $\alpha \in \Delta^3$ respectively. The minimum distance local bi-level formulation result, initialized with $\alpha_i = 0.2$, is annotated as $\textrm{ProjGM}$ and shown, alongside the Keshavaraz minimum KKT-constraint violation formulation \cite{keshavarz2011imputing}. The $M_{x1}$ entry of the PSD matrix output by algorithm \ref{alg:uncons_lmi} is also shown.

\begin{figure}[h]
    \centering
    \includegraphics[width=0.75\linewidth]{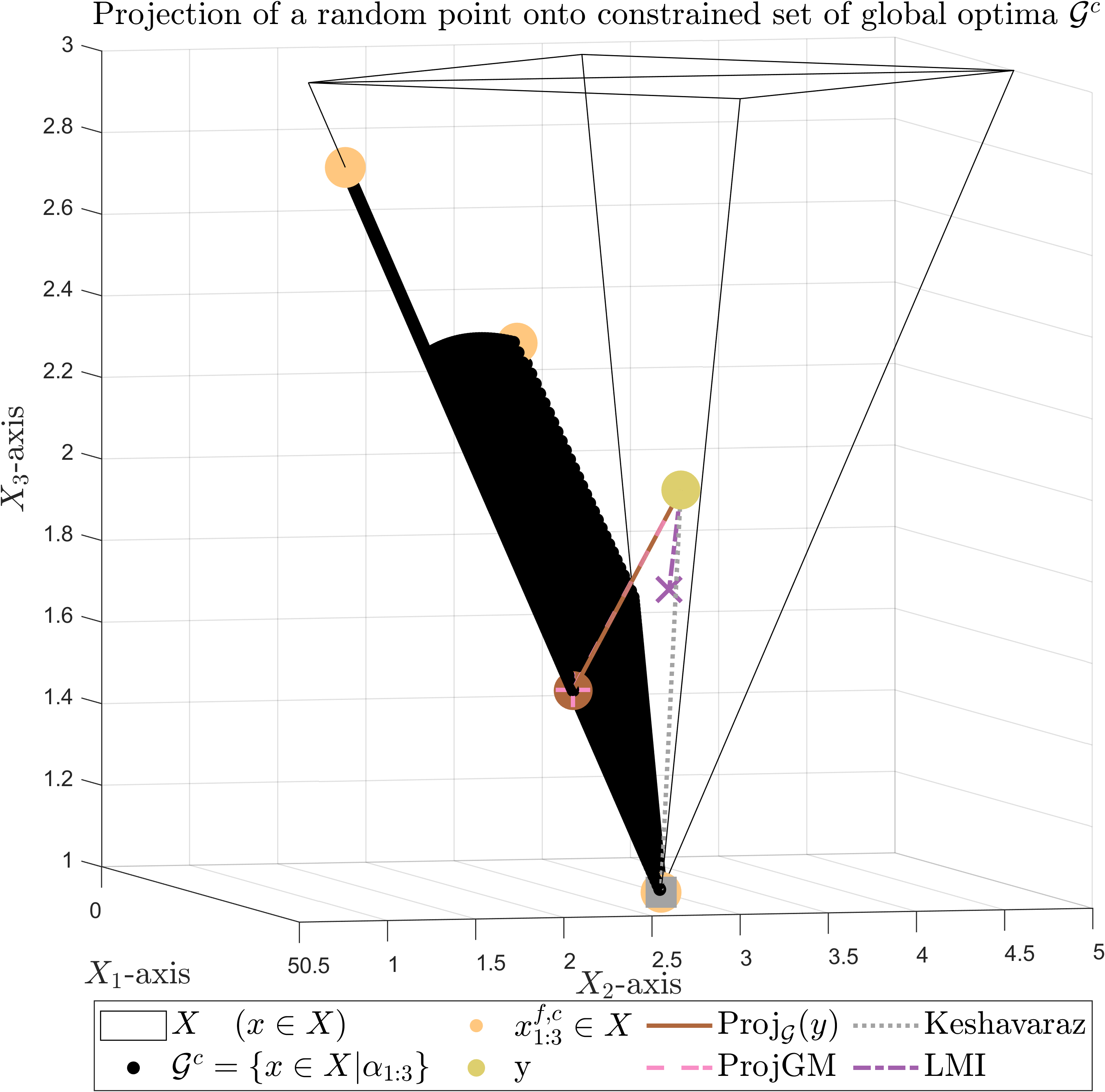}
    \caption{Comparison of the projection onto a constrained set of optima $\gs^c$.}
    \label{fig:proj_to_cons}
\end{figure}

From figures \ref{fig:proj_to_uncons} and \ref{fig:proj_to_cons} it is possible to see that the approximate Keshavaraz method can impute a strongly inconsistent objective if the test point is outside the set of global optima $\gs$. The LMI algorithms \ref{alg:uncons_lmi} and \ref{alg:cons_lmi}, are designed to invalidate a proposed set of basis functions with respect to a data set. In other words, the algorithm is able to say whether IOC with proposed basis functions can be successfully applied to a given data set.

Given a data set of a $100$ points $y \in \R^2$ generated by adding Gaussian noise to the shown test point in Figure \ref{fig:proj_to_uncons}, the distance of each of them to the set of global optima is calculated via different methods. The mean distance to the set of global optima for different methods is, in ascending order, $d_{\rm LMI} = 3.78$, $d_{\rm ProjGM} = 4.09$ and $d_{\rm Keshavaraz} = 6.98$, the true value being $d_{Proj_{\gs}(y)} = 4.09$.

In practice, for higher-dimensional data, running the bi-level method may be expensive. If the Keshavaraz method does not produce a zero or extremely-close-to-zero error, it is useful to be able to run the LMI algorithm to check the lower-bound of the distance from the set of global optima. If the lower bound is higher than the desired error in data replication, we can conclude that the proposed set of basis functions $\fs$ is not good enough to represent the data.

MATLAB (2021b) code to replicate figures and experiments is publicly available at 
\url{https://github.com/jarmill/inverse-optimal}. 
Dependencies include Mosek \cite{mosek92} and YALMIP \cite{lofberg2004yalmip}.
\section{Extensions} \label{sec:extensions}
This section will detail extensions of the current work and some discussion of future research directions.

\subsection{Piecewise Functions and Convex Lifts}
The method of convex lifts \cite{yannakakis1991expressing, gouveia2013lifts} may be applied to solve the \ac{pgm} problem over some classes of piecewise-defined costs $f_j$. Assume there exists a set of $C^1$ functions $\{w_{j \ell}(x)\}_{\ell=1}^{L_j}$ with ${L_j}$ finite such that the convex $f_j$ satisfies,
\begin{equation}
\label{eq:pw_max}
    f_j(x) = \max_{\ell \in 1 .. L_j} w_{j \ell}(x).
\end{equation}

\begin{rem}
Convexity is preserved under the pointwise maximum operation (though not under pointwise minimum). 
\end{rem}
\begin{rem}
The class of convex piecewise affine functions with $w_{j \ell}(x) = m_{j \ell}^T x - b_{j \ell}$ for $m_{j\ell} \in \R^{n \times 1}, b_{j \ell} \in \R$ may be expressed as an instance of \eqref{eq:pw_max}.
\end{rem}
New variables $\tau_j$ may be added for every max-representable function $f_j(x)$, forming the equivalent optimization problems with equal objectives:
\begin{align}
    f^* =& \min_{x \in X} \textstyle \sum_{j=1}^m f_j(x) \label{eq:pw_nondiff}\\
    f^* =& \min_{x \in X, \ \tau \in \R^m} \textstyle \sum_{j=1}^m \tau_j \label{eq:pw_diff} \\
    & \alpha_j w_{j \ell}(x) \leq \tau_j & \forall \ell=1..L_j, \ j = 1..m. \nonumber 
\end{align}

The formerly non-differentiable objectives $f_j(x)$ from \eqref{eq:pw_nondiff} are lifted into the constrained problem \eqref{eq:pw_diff} where each $w_{j \ell}(x)$ is differentiable for all $x \in X$. \ac{KKT} equations may then be written for \eqref{eq:pw_diff} and then utilized in \eqref{eq:gh_cons} to describe the cone $\gh^c(x, \tau, \alpha, \mu, \lambda)$ for use in the constrained \ac{pgm} problem \ref{prob:pgm_cons}.

\begin{rem}
Minimization of the $L_1$ norm $\norm{x}_1$ may be expressed as $\min \sum_{i=1}^n t_i: \ -t_i \leq x_i \leq t_i$ in $2n$ inequality constraints by adding $n$ new variables $\{t_i\}_{i=1}^n$ for use in constrained \ac{pgm}.
$L_p$ norms with rational $p \in [1, \infty)$ admit second-order-cone representations \cite{alizadeh2003second} through lifting, and can therefore be members of the dictionary $\fs$.
\end{rem}

A rank-1 solution $M$ of \eqref{eq:uncons_sdp} certifies that $p^*_{low} = p^*$ from \eqref{eq:pgm_uncons}, and the optimal $(x, \alpha)$ may be read from $M_{x 1}$ and $M_{\alpha 1}$

\subsection{Projection onto Local Minimizers}
This paper was restricted to convex cost functions in $\fs$ and convex sets $X$, with an additional polynomial requirement for \ac{pgm}. Problem \ref{prob:pgm} may be extended to finding the minimum distance $p^*$ between $y$ and some local minimizer of $f_\alpha(x)$. Theorem 12.6 (Eq. (12.65)) of \cite{wright1999numerical} outlines second-order necessary conditions for a point $x^*$ to be a local minimizer. In the unconstrained case, the Hessian matrix $\nabla^2_x f_\alpha(x)$ must be \ac{PD} in addition to $(x, \alpha)$ satisfying the first-order condition $\nabla_x f_\alpha(x) = 0$. The constrained case with non-convex functions requires more delicacy, as the quadratic form $w \rightarrow w^T \nabla^2_x f_\alpha(x) w$ must be positive for all $w \neq 0$ vectors inside the Critical Cone (Eq. (12.53) of \cite{wright1999numerical}) formed by $(x, \lambda, \mu)$. 

\subsection{Polynomial Optimization}
The \ac{pgm} programs \ref{prob:pgm_uncons} and \ref{prob:pgm_cons} are instances of Polynomial Optimization Problems when $f_j(x), g_k(x)$ are all convex polynomial functions of $x$. The sets $\gh(x, \alpha), \gh^c(x, \alpha, \mu, \lambda)$ are basic semialgebraic sets, and their projections $\gs(x), \gs^c(x)$ are in turn semialgebraic sets. The projections $\gs(x), \gs^c(x)$ may be analyzed by quantifier elimination algorithms such as the Cylindrical Algebraic Decomposition \cite{caviness2012quantifier}. The moment-\ac{SOS} hierarchy is a method that yields a rising sequence of lower bounds to the distances $p^*$ by solving a sequence of \acp{LMI} in combinatorially increasing size \cite{lasserre2009moments}. Theorem \ref{thm:compact} assures (under mild conditions) that the moment-\ac{SOS} relaxations of the unconstrained problem \ref{prob:pgm_uncons} will converge at a finite degree. Convergence of the moment-SOS hierarchy for bilevel Polynomial Optimization Problems are established in \cite{jeyakumar2016convergent}.
The \ac{LMI} presented in Section \ref{sec:lower_bounds} is an instance of the degree-1 moment-\ac{SOS} hierarchy as enriched with valid constraints. 

Given arbitrary convex polynomials $f_j(x), g_k(x)$, all constraints in $\gh(x, \alpha)$ are affine w.r.t. $\alpha$. Likewise, the describing constraints of $\gh^c(x, \alpha, \mu, \lambda)$ are affine in $(\alpha, \mu, \lambda)$ all together. A theorem of alternatives may be used to eliminate the affine-dependent groups $(\alpha)$ or $(\alpha, \mu, \lambda)$, yielding a set of linked \ac{LMI} constraints of smaller size that solely depend on $x$ \cite{ben2015deriving}. This reduction in the number of variables decreases the computational burden of solving \acp{LMI} as the degree increases. A full presentation and application of polynomial optimization for \ac{pgm} will take place in sequel work.
\section{Conclusion} \label{sec:conclusion}
This paper introduced a cost functions set invalidation interpretation of the \ac{pgm} problem. The geometry of optimal sets were explored in the constrained and unconstrained cases (focusing on convex quadratics). Numerical algorithms were implemented to upper and lower bound the \ac{pgm} optima.
Polynomial and convexity requirements are imposed for the \ac{pgm} problem, but possible extensions have been mentioned. 
Future work includes performing cost-function discovery in order to generate candidate functions $f(x)$ that would reduce the distance $p^*$ from \eqref{eq:pgm} when added to the set of cost functions $\fs$. 

Application to simple mechanical models where constraints are convex is possible, but for more complex mechanical systems extensions of this framework are required. Therefore, the proposed method should be applied for convex polynomial IOC problems based actual noisy human data. 

On the theoretical side, sensitivity methods from set valued analysis will be applied to determine which properties of $\gs$ and $\gh$ are preserved when assumptions are lifted (e.g. does Theorem \ref{thm:path_quad} hold when all $f_j$ are strongly convex rather than strongly convex quadratic?). Continuity properties of the solution map \eqref{eq:qp_argmin} and its selection \eqref{eq:y_selection} will also be explored \cite{aubin2009set, tam1999continuity}.
\section*{Acknowledgements}

F. Bečanović would like to thank Florent Lamiraux for his useful remarks.
J. Miller would like to thank his advisors Mario Sznaier and Didier Henrion, and the POP group at LAAS-CNRS for their comments and support.
\balance
\bibliographystyle{IEEEtran}
\bibliography{references.bib}
\end{document}